\documentclass[a4paper,reqno,12pt]{amsart}
\usepackage{graphicx}
\usepackage{pslatex}
\usepackage{mathtools}
\usepackage{pgfplots}
\pgfplotsset{compat=1.15}
\usepackage{mathrsfs}
\usetikzlibrary{arrows}
\usepackage{amsmath,amsfonts,amssymb, amsthm}
\usepackage{mathrsfs}
\usepackage{rotating}
\usepackage{amssymb}
\usepackage{verbatim}
\usepackage{color}
\usepackage[utf8]{inputenc}
\usepackage{hyperref, latexsym}
\usepackage{url}

\newtheorem{theorem}{Theorem}
\newtheorem{conjecture}{Conjecture}
\newtheorem{lemma}[theorem]{Lemma}

\newtheorem{corollary}[theorem]{Corollary}
\newtheorem{remark}{Remark}
\newtheorem{definition}{Definition}

\def\R{\mathbb R}
\def\H{\mathbb H}

\def\supp{\text{supp}}
\def\({\left(}
\def\){\right)}
\def\[{\left[}
\def\]{\right]}
\def\<{\left<}
\def\>{\right>}

\usepackage{fullpage} 
\usepackage{setspace}
\usepackage{adjustbox}

\usepackage{mathtools}
\mathtoolsset{showonlyrefs}

\def\today{\ifcase\month\or
  January\or February\or March\or April\or May\or June\or
  July\or August\or September\or October\or November\or December\fi
  \space\number\day, \number\year}

\newcommand{\A}{\mathcal{A}}
\newcommand{\B}{\mathcal{B}_{\infty}}

\newcommand{\D}{\mathbb{D}}
\newcommand{\E}{\mathcal{E}}

\newcommand{\G}{\mathcal{G}}
\renewcommand{\H}{\mathbb{H}}

\newcommand{\z}{\mathbb{Z}}

\renewcommand{\r}{\mathbb{R}}
\newcommand{\cp}{\mathbb{C}} 
\newcommand{\im}{{\rm Im}\,}
\newcommand{\re}{{\rm Re}\,}

\newcommand{\wt}{\widetilde}

\newcommand{\ga}{\gamma}
\newcommand{\al}{\alpha}
\newcommand{\be}{\beta}
\newcommand{\ep}{\varepsilon}

\newcommand{\si}{\sigma}

\renewcommand{\d}{\mathrm{d}}

\newcommand{\p}{\varphi}
\renewcommand{\ga}{\gamma}

\newcommand{\bl}{{BL}_+(\R)}
\renewcommand{\b}{{\rm Hom}(\H)}
\newcommand{\id}{{\rm Id}}
\renewcommand{\z}{\overline{z}}
\renewcommand{\p}{\varphi}
\newcommand{\tobl}{\xrightarrow{\bl}}
\newcommand{\aff}{{\rm Aff}(\H)}
\newcommand{\La}{\Lambda}
\renewcommand{\supp}{{\rm supp}}

\begin{document}
\title{On the Uniqueness of the Norton-Sullivan Quasiconformal extension}

\author{Jose A. Barrionuevo, Felipe Gonçalves, Jose Victor Medeiros and Lucas Oliveira}

\address{Departamento de Matem\'{a}tica - UFRGS, Porto Alegre, Brazil, 91509-900}
\email{josea@mat.ufrgs.br}

\address{IMPA - Instituto de Matemática Pura e Aplicada, Rio de Janeiro, 22460-320, Brazil.}
\email{goncalves@impa.br}

\address{IIMPA - Instituto de Matemática Pura e Aplicada, Rio de Janeiro, 22460-320, Brazil.}
\email{victor@impa.br}

\address{Departamento de Matem\'{a}tica - UFRGS, Porto Alegre, Brazil, 91509-900}
\email{lucas.oliveira@ufrgs.br}

\date{\today}

\maketitle


\begin{abstract}
We show that the extension map 
$$
\mathcal{E}_{NS}(f)(z)=\frac{f(x+y)+f(x-y)}{2}+i\frac{f(x+y)-f(x-y)}{2}
$$
{for all } $z=x+iy\in\H$, defined by Norton and Sullivan in '96, is the only locally linear extension map taking bi-Lipschitz functions on $\R$ to quasiconformal functions on $\H$, modulo the action of a group isomorphic to the linear group. In fact, we discovered many other extensions like this one (lying in the orbit of such group action), such as: $f(x)\mapsto f(x)+i(f(x)-f(x-y))$.
\end{abstract}

\section{Introduction}

\subsection{Background}
Around the fifties, one important question in the theory of quasiconformal mappings was to decide the regularity of such mappings when acting on the boundary of a domain. In the complex plane, the simplest situation can be reduced, by the Riemann Mapping Theorem, to the boundary value of quasiconformal automorphisms of the upper half-plane $\H$ or the unit disk $\D$, depending on the framework . The goal was \emph{to characterize boundary homeomorphisms of the real line induced by quasiconformal homeomorphisms of the upper half-plane.} In their influential article \cite{Beurling-Ahlfors}, Ahlfors and Beurling completely classified the class of homeomorphisms of the real line that preserves orientation and that could appear as boundary values of quasiconformal mappings: such mappings are now called \emph{quasisymmetric}, and they admit the following characterization in terms of the $M$-condition: 
We say that a homeomorphism $f:\R\to\R$ verifies the $M$-condition if there is $M\ge1$ such that 
\begin{equation}\label{quasisymmetry}
    \frac{1}{M}\le \frac{f(x+t)-f(x)}{f(x)-f(x-t)}\le M \quad x\in \R \text{ and } t>0.
\end{equation}
We denote by $QS(\R)$ the class of all homeomorphisms that verify the $M$-condition, and we denote by $QC(\mathbb{H})$ the collection of all {quasiconformal\footnote{A homeomorphism $F:\H\to\H$ is quasiconformal if $F\in W^{1,2}_{loc}(\H)$ and if $\sup_{z\in \H} \frac{|\partial_{\overline{z}} F(z)|}{|\partial_z F(z)|}<1$.} automorphisms of the upper half-plane $\H:=\{z=x+iy: x\in \R, \, y>0\}$}. In fact, as part of their characterization, Ahlfors and Beurling showed that any $f\in QS(\R)$ can be extended to a quasiconformal mapping $\E_{AB}(f):\mathbb{H}\to \mathbb{H}$ given by
\begin{equation}\label{beurling-ahlfors-ext}
\E_{AB}(f)(z)=\frac{1}{2}\int_{-1}^{1} f(x+ty)dt+\frac{i}2\int_{-1}^{1} f(x+ty)\text{sign}(t)dt\,.
\end{equation}

In light of the above results, one could go further and ask if it is possible to produce an extension process that has extra properties, often needed in applications. We let $\aff$ denote the set of affine transformations $f(z)=az+b$ with $a>0$. Then a simple computation shows that $\E_{AB}(f|_{\R})=f$ for all $f\in \aff$ and that
\begin{equation}
\E_{AB}(f\circ g)=\E_{AB}(f)\circ \E_{AB}(g)   \quad  \text{ if } \  f\in \aff \  \text{ or } \ g\in \aff.
\end{equation} 
For applications in complex dynamics, for example, it is desirable to have an extension process that extends quasisymmetric homeomorphisms to quasiconformal homeomorphisms acting on the unit disk. This is given by the Douady and Earle extension process \cite{Douady-Earle}. In precise terms, they construct an extension map $\E_{DE}: QS(\mathbb{S})\to QC(\mathbb{D})$ such that $\E_{DE}(m|_{\mathbb{S}})=m$ for any Möbius automorphism and
$$
\E_{DE}(f\circ g)=\E_{DE} f \circ \E_{DE} g  \quad  \text{ if } \  f\  \text{ or } \ g \ \text{ is a Möbius transformation} .
$$
An extension with this property is called \emph{conformally natural}. Precisely, $\E_{DE}(f)(z)=w$ exactly when
\begin{equation}\label{douady-earle-ext}
\int_{\mathbb{S}^{1}}\frac{w-f(\zeta)}{1-\overline{w}f(\zeta)}\frac{|d\zeta|}{|\zeta-z|^2} = 0
\end{equation}
(it can be shown that for any $z\in \D$ the above equation has a unique solution $w\in \D$).  More details about such extensions can be found in \cite{Lehto-Virtanen, Hubbard-book1, Earle}.

We can go one step further and ask if we could impose that the group operation is preserved by the extension process, that is, composition in the domain is mapped into composition in the image; this is what Sullivan called the \emph{Dream Problem},  stated explicitly in \cite{Norton-Sullivan}:

\subsection*{Dream Problem} \emph{Is it possible to construct an extension $\mathcal{E}:QS \to QC$ such that $\mathcal{E}(f\circ g)=\mathcal{E}(f)\circ \mathcal{E}(f)$ {for all } $f,g\in QS$?}
\smallskip

 The answer was shown to be in the negative by Epstein and Markovic \cite{Epstein-Markovic}, their \emph{Stop Dreaming Theorem}, by employing a combination of diverse techniques (ranging from the Baire Category Theorem and intricate computations involving the affine group and earthquake functions) to obtain a beautiful proof by contradiction.  Besides the negative answer to the \emph{Dream Problem}, one could perhaps continue dreaming by replacing the $QS$ class, by some smaller group of functions. For instance, by the results of Hinkkanen \cite{Hinkkanen-book} and Markovic \cite{Markovic}, if we restrict ourselves to \emph{uniformly quasisymmetric groups}, that is, subgroups $\G\subseteq QS(\r)$ for which each member verifies the same $M$-condition, the Dream Problem has a positive solution. The reason for such extension to exist is that any uniformly quasisymmetric group $\G$ is quasi-symmetrically conjugated to some Möbius group. 
 
 On another direction, if we drop the uniform control of the quasisymmetric constant and replace $QS(\r)$ by another subgroup of $QS(\R)$, for instance the group $BLip(\R)\subsetneq QS(\R)$ of bi-Lipschitz homeomorphisms of the real line\footnote{A homeomorphism $f:\R \to\R$ is bi-Lipschitz if there is $c\geq 1$ such that $c^{-1}|x-y|\leq |f(x)-f(y)|\leq c|x-y|$ for all $x,y\in \R$}, then we have a positive result: Norton and Sullivan \cite{Norton-Sullivan} provides us with an explicit extension that does the job
    \begin{equation}\label{triang-ext}
        \mathcal{E}_{NS}(f)(z)=\frac{f(x+y)+f(x-y)}{2}+i\frac{f(x+y)-f(x-y)}{2}\mbox{ for all }z=x+iy\in\H\,.
    \end{equation}
    It is not hard to verify that $\mathcal{E}_{NS}: BLip(\R) \to QC(\H)$ and $\E_{NS}(f\circ g) = \E_{NS} f \circ \E_{NS} g$.
These results, when combined with the negative results of Epstein and Markovic, and with the fact that there are examples of homeomorphisms $f\in QS(\r)\backslash BLip(\R)$ for which $\mathcal{E}_{NS}(f)\notin QC(\H)$ (a simple counter-example is provided by  $f(x)=x^{3}$) {leave open the following simple question: 

\begin{center}\emph{Is there any group $G$ with $BLip(\R)\subsetneq G \subsetneq QS(\R)$ such that  $\mathcal{E}_{NS}:G\to QC$? }\end{center}

Surprisingly the answer is no! To be precise, if $f:\R \to \R$ is a homeomorphism such that $E_{NS}(f)$ is in $QC(\H )$ then $f$ is bi-Lipschitz. Besides, it is in general false that quasiconformality of mappings $F:\H\to\H$ implies absolute continuity of their boundary values; this is true only in our special case. The proof is quite simple:
\begin{itemize}
	\item By hypothesis $E_{NS}(f)$ is K-quasiconformal for some $K\ge1$;
    \item The quasiconformality  ensures that $E_{NS}(f)$ is absolutely continuous on almost every horizontal and vertical line. Due to the simple nature of $E_{NS}$ this immediately implies that $f $ is absolutely continuous. We may assume $f$ to be increasing  and that $f(\pm \infty ) = \pm\infty$. Thus,  there is  $c_0>0$ such that $ | \{ s \in \R : |f^{\prime}(s) - c_0| < \epsilon \} |>0$ for all $\ep>0$. 
    \item We now claim that $f' \in L^\infty(\R)$. The proof is by contradiction: {if the claim is false, either $0$ or $+\infty$ are in the essential range of $f'$}. In the first case we choose 
$0 < \epsilon_0 < c_0/2$, let $A_0 = \{ s\in \R : |f^{\prime}(s) - c_0| < \epsilon_0\}$ be of positive measure, and for $\delta > 0$, to be chosen later, we set $B_{\delta} = \{ t\in\R : 0 < f'(t) < \delta\} $, which also has positive measure. For $(s,t) \in A_0 \times B_{\delta}$ there exist unique $x$ and $y$ such that $x-y = t,\; x+y = s$. If $\delta$ is sufficiently small we can make
$$ \left| \frac{f^{\prime}(x+y) - f^{\prime}(x-y)}{f^{\prime}(x+y) + f^{\prime}(x-y)}\right| > \frac{K-1}{K+1} $$
for $(x,y)$ on a set of positive measure, contradicting the $K$-quasiconformality hypothesis. The case when the function assumes arbitrarily large values can be treated similarly.
\end{itemize}

The curious answer to the question above lead us to speculate if there is any other extension process similar to the extension of Norton and Sullivan:}

\subsection*{Problem  (Main Question)}
\emph{Is there any other homomorphic extension ${E}:BLip(\R)\to QC(\H)$ distinct from $\E_{NS}$?}

It turns out, as our main result is going to prove, that  if one assumes the extension $E$ has the main properties that $\E_{NS}$ has, that is, $E$ is similar to $\E_{NS}$ in some way (to be defined below), then the answer is \textbf{no!}. That is, after symmetry, $E$ coincides with $\E_{NS}$

\subsection{Main Results} 
We are now about to enunciate a classification result which states (roughly) that any extension $E$ satisfying certain properties must be, after symmetries, equal to $\E_{NS}$. Thus, the smaller the domain where $E$ is defined and the larger the codomain of $E$, the stronger our results. This is why below we define certain spaces, which might not be usual to work with in the literature, but nevertheless, represent the extent which our results hold true (compare Theorem \ref{thm1} and Corollary \ref{cor1}).

We fix orientation and let $\bl$ denote the class of bi-Lipschitz increasing $C^1$-functions. That is, $f\in \bl$ if and only if $f\in C^1(\r)$ and there are $b,B>0$ such that $b\leq f'(x)\leq B$ for all $x\in \r$. Note $\bl\subset BLip(\R)$. Let $\{f_n\}$ be a sequence in $\bl$ and $f\in \bl$.   We write $f_n \tobl  f$ if $f_n\to f$ and $f_n'\to f'$ uniformly in compact sets of $\r$, and  there are $b,B>0$ such that $b\leq f_n'(x) \leq B$ for all $n\geq 1$ and $x\in \r$. Note that if $f_n \tobl  f$ then $f\in BLip(\R)$ and $f_n\to f$ uniformly in compact sets of $\R$. Let $\b$ be the class of homeomorphisms of the upper half-space $\H:=\{z=x+iy : y>0\}$.

\begin{definition} We say a map $E: \bl \to\b$ is an {\bf admissible extension} if:
\begin{enumerate}
\item[(P1)] $\lim_{y\to 0}  Ef(x+iy) = f(x)$ uniformly in compact sets of $\r$;

\item[(P2)] $E(a\id +b)(z)=az+b$ for every $a>0$ and $b\in \r$;

\item[(P3)] $E(f \circ g)=Ef \circ Eg$ for all $f,g\in\bl$;

\item[(P4)] If $f_n \tobl f$  then
$
Ef_n(z)\to Ef(z)
$
{for all} $z\in \H$.
\end{enumerate}
We write $E\in \A$.
\end{definition}

These are, in a way, the minimal conditions for an extension $E$ to satisfy the \emph{Dream Problem} with $QS(\R)$ replaced by $\bl$.  Consider the non-commutative group $G=\r\times \r_+$ with the following group operation
$$
(a,\al)\cdot (b,\be) = (a+\al b,\al\be).
$$
Observe that $(a,\al)^{-1}=(-a/\al,1/\al)$ and that $(0,1)$ is the neutral element of $G$ and that 
$$
(a,\al)\in G\mapsto a+\al z \in \aff
$$
is an isomorphism. There is an action of $G$ over $\A$ given by
\begin{align}
(a,\al)E f(x+iy) = S_{-a/\al}^{1/\al} \circ Ef \circ S^\al_a (z)
\end{align}
for any $E\in\A$, where $S^\al_a (z) =x+ay+i\al y$. Routine calculations show that
$$
(a,\al)E \in \A, \quad  (a,\al)((b,\be) E) = ((a,\al)(b,\al))E , \ \text{ and } \ (0,1)E=E,
$$
for any $E\in \A$. Using this action, we can generate an entire family of extensions in the spirit of Norton and Sullivan \cite{Norton-Sullivan}; in particular, the construction below shows directly that  $\E_{NS}$ is member of a family of transformations, which we define below. It seems that this group action was not noticed before, nor the existence of the family of transformations we define below.

\begin{definition}
Consider the two-parameter family of extensions $\E_{a,\al}:\bl \to\b$ defined for each
$f\in \bl$ by 
$$
\E_{a,\al}f(z) = (1-a/\al)f(x+a y) + (a/\al)f(x-(\al-a) y) + \frac{i}{\al}f(x+a y) - \frac{i}{\al}f(x-(\al-a) y),
$$
for $a\in \r$ and $\al> 0$. When $\al=0$ we interpret the above family of operators as taking the limiting case $\alpha\to 0$ in the above expression, which gives us the formula
$$
\E_{a,0}f(z) = f(x+ay)-ayf'(x+ay) + iyf'(x+ay).
$$
\end{definition}

It is straightforward to show that $\E_{a,\al}\in \A$. Observe that 
$$
\E_{a,\al}=(a,\al)\E_{0,1}
$$
where 
$$
\E_{0,1}f(z) = f(x) + i(f(x)-f(x-y)).
$$
Note that $\E_{1,2}=\E_{NS}$.  Observe also that $\E_{a,0}=(a,0)\E_{0,0}$, where $\E_{0,0}=f(x)+iy f'(x)$.  In fact, $\E_{a,\al}$ is not only an admissible extension, but $F(z)=\E_{a,\al}f(z)$ is always a quasiconformal map of $\H$ for $\al>0$. Indeed, quasiconformality follows from the computation
{\begin{align}\label{eq:Fquasiconf-est-alt}
\left| \frac{\partial_{\z} F(z)}{\partial_z F(z)}  \right| =  \left|\frac{1-\theta}{1-e^{i\sigma}\theta}\right| ,
\end{align}
where $\theta=f'(x-(\al-a) y)/f'(x+ay)$ and 
$$
e^{i\sigma}=\frac{(-i+a)(i+a-\alpha)}{(i+a)(-i+a-\alpha)}.
$$
Since $\theta \in [1/c,c]$ for some $c>1$ and $e^{i\sigma}=1$ only happens when $\alpha=0$, we have $\left|\frac{1-\theta}{1-e^{i\sigma}\theta}\right|<1-\ep$, hence quasiconformality follows. On the other hand, the transformation $F=\E_{0,0}f$ is not quasiconformal (assuming $f\in C^2$) since
$$
\sup_{z\in \H}\left| \frac{\partial_{\z} F(z)}{\partial_z F(z)}  \right| = \sup_{z\in \H} \frac{|f''(x)y|}{|2f'(x)+if''(x)y|} = 1
$$
when $f$ is not affine.\footnote{It is interesting to observe that such extensions are quasiconformal for $\alpha>0$ if and only if $f$ is bi-Lipschitz.}

In what follows, $C^\infty_c(\r)$ is the class of infinitely differentiable functions $\p:\r\to\r$ with compact support and $C^{1}_0(\r)$ is the class of $C^1$-functions $\p:\r\to\r$ such that both $\p$ and $\p'$ converge to zero at $\pm \infty$. It is classical that $C^{1}_0(\r)$ is a Banach space with the norm $\|\p\|_{1,\infty} = \|\p\|_{\infty}  + \|\p'\|_{\infty}$. We also let $\B(\ep)=\{\p\in C^\infty_c(\r): \|\p'\|_{\infty}<\ep\}$.

The following are the two main results of this article.

\begin{theorem}\label{thm1}
Let $E\in \A$. Assume there is $z_0\in \H$, $\ep_0\in (0,1)$ and a linear functional $\Lambda: C^\infty_c(\r) \to \mathbb{C}$ such that
\begin{equation}\label{linear_at_id}
E(\id+\p)(z_0)=z_0 + \Lambda(\p)
\end{equation}
for any $\p \in \B(\ep_0)$. Then $E=\E_{a,\al}$ for some $a\in \r$ and $\al\geq 0$.
\end{theorem}

\begin{theorem}\label{thm2}
Let $E\in \A$. Assume there is $z_0\in \H$, $\ep_0\in (0,1)$ and a bounded linear functional $\Lambda: C^1_0(\r) \to \mathbb{C}$ such that
\begin{align}\label{linear_at_idweak}
{E(\id +\p)(z_0)=z_0+\Lambda(\p)} + R(\p),
\end{align}
for any $\p \in \B(\ep_0)$, where $R:\B(\ep_0)\to\cp$ is a function satisfying the following property: For any $c>0$ there is $\ep\in (0,\ep_0)$ such that if $\p \in \B(\ep)$ then
$$
|{R(\p)}| \leq c \|\p'\|_{\infty}^2.
$$
Then $\Lambda=\E_{a,\al}$ for some $a\in \r$ and $\al\geq 0$.
\end{theorem}

The reason we use the ball $\B(\ep_0)$ and not a ball in $C^1_0(\r)$, is that it makes the theorem stronger.

One way to think about both results above is that, in the first, we ask the extension to be \emph{identically linear} in a neighborhood of the identity map, while in the second, we only ask that the extension has a continuous (variational) derivative at the identity map, with a \emph{vanishing} second derivative. Indeed, the condition we ask on the remainder $R$ in Theorem \ref{thm2} can be interpreted heuristically as:
$$
DE|_{\id} \ \text{ is continuous} \quad \& \quad D^2 E|_{\id}\equiv 0. 
$$
We then proceed to show that $DE|_{\id} = \E_{a,\al}$

\begin{corollary}\label{cor1}
    Let $E:BLip(\R)\to QC(\H)$ be an extension verifying (P1), (P2) and (P3), and the following condition
   \begin{itemize}
   \item[$(P4)^*$]  If $f_n \to f$ uniformly in compact sets of $\R$, then $E f_n(z) \to E f(z)$ {for all} $z\in \H$.
\end{itemize}
    Assume also that $E$ verifies  condition \eqref{linear_at_id} or \eqref{linear_at_idweak}.  Then the conclusions of Theorems \ref{thm1} or Theorem  \ref{thm2} follow (respectively).
\end{corollary}
\begin{proof}
We have $\bl\subset BLip(\R)$ and $QC(\H) \subset \b$. Moreover, if $f_n\tobl f$ then $f_n\to f$ uniformly in compact sets. We conclude that $E \in \A$. We can now apply Theorems \ref{thm1} and \ref{thm2} respectively. 
\end{proof}

\begin{remark}\label{rem: z_0}
In Theorems \ref{thm1} and \ref{thm2}, we can always assume that $z_0=i$. Indeed, define the operators $T_{x_0}f(x)=f(x-x_0)$ and $S_{y_0}f(x)=f(y_0x)$. Let $f(x)=x+\p(x)$ for $\p\in\B(\ep_0)$. Observe that by properties (P2) and (P3) we have 
\begin{align}
Ef(i)=E(T_{x_0}S_{1/y_0}f)(z_0)  & = E[(T_{x_0}S_{1/y_0}\id) \circ (\id + y_0 T_{x_0}S_{1/y_0}\p)](z_0) \\ & = y_0^{-1}\left(E[\id + y_0 T_{x_0}S_{1/y_0}\p](z_0)-x_0\right) \\
& = \Lambda(T_{x_0}S_{1/y_0}\p)+i.
\end{align}
Notice the functional $\p\mapsto \Lambda(T_{x_0}S_{1/y_0}\p)$ still is linear.
\end{remark}

Despite many efforts, we were unable to employ our techniques to prove the following conjecture, which we believe is true:

\begin{conjecture}[{\bf The Brazilian Dream Problem}]\label{conj}
Let $E: \bl \to QC(\H)$ be an extension satisfying properties (P1), (P2), (P3) and (P4). Then 
$$
E = \E_{a,\al}
$$
for some $a\in \r$ and $\al> 0$.
\end{conjecture}

\subsection{About the proof of the main results}
The proof of our main results will take several steps and, besides using classical techniques, they might be unusual for this field, so we will make an effort to explain them as best as possible. 
\begin{itemize}
    \item [(i)] We will obtain, using a limiting process and some classical arguments, that any $E\in \A$ admits a representation at the identity in terms of two Radon measures $\mu$ and $\nu$ (see \eqref{Enewform}), to be  specified later;
    \item [(ii)] Later on, using properties (P2) and (P3) together with some perturbations of the identity in $\bl$, we get that $\mu$ and $\nu$ must satisfy a functional identity (see \eqref{eq:varidentity}), which lead us to conclude (crucially) that $\mu$ and $\nu$ have to be supported in at most $2$ points;
    \item  [(iii)] Finally, after several considerations on the equations, the support and  weights that the measures $\mu$ and $\nu$ are allowed to have, we will conclude that  $E=\E_{a,\al}$.
\end{itemize}
In the third section we will give a walk-through explanation of the main results assuming that the extension is itself linear: this captures many of the important steps and serves to motivate the general case. 

\subsection{Organization} In {Section \ref{sec2}} we will show how to reduce the proof of the {Theorem \ref{thm1}} to the proof of {Theorem} \ref{thm2};
     In {Section} \ref{sec3} we will sketch a proof of a special case of our results (assuming that the extension is linear) in order to motivate the general line of reasoning; In Section \ref{sec4} we present the full proof; Finally in {Section} \ref{sec5} we will collect some remarks, comments and questions about such extensions .

\section{Reduction of Theorem \ref{thm1} to Theorem \ref{thm2}}\label{sec2}

The lemma below seems to be known, though perhaps not in this exact form. Since we were unable to find a direct reference for it, we have included a proof for clarity.

\begin{lemma}\label{lem}
Let $E_1,E_2\in \A$ and assume there is $\ep_0 \in (0,1)$ such that $E_1(f)=E_2(f)$ for all $f\in \bl$ such that $\|f'-1\|_\infty< \ep_0$. Then {$E_1=E_2$}.
\end{lemma}
\begin{proof}
It is enough to show that for every $f \in \bl$ there is $f_1,f_2,...,f_N \in \bl$ such that $\|f_j'-1\|_\infty <\ep_0$ for each $j$ and $f=f_1\circ f_2\circ \cdots \circ f_N$. Given $f \in \bl$, we know that there is $L\ge1$ such that $L^{-1}\le\|f'\|_\infty\le L$. We can assume that $L>1+\ep_0$, otherwise there is nothing to prove. Moreover, we can also assume that $f(0)=0$ (the general case being reduced to this one by a translation). Choose $\ep<\ep_0$ so small that $\ep_0>\ep/(1-\ep)$ and define $\alpha_{1}=\log(1+\ep)/\log L$. If we consider 
$$
f_{1}(x)=\int_{0}^{x}f'(t)^{\alpha_1}dt,
$$
we see by the construction of $f_1$ that $(1+\ep)^{-1}\le\|f_{1}'\|\le 1+\ep$ what gives, in particular, that $\|f_{1}'-1\|<1+\ep_0$. Define $g_{1}=f\circ f_{1}^{-1}$ and observe that by the chain rule $L_{1}^{-1}\le\|g_{1}'\|_\infty\le L_{1}$, where $L_1=L/(1+\ep)$. Repeat this process until $L_{N}=L/(1+\ep)^{N}<(1+\ep_0)$ to obtain a decomposition $f=f_{N+1}\circ f_{N}\circ\cdots\circ f_1$ (with $f_{N+1}=g_{N}$) where $f_{j}$ verifies $\|f_j'-1\|_\infty <\ep_0$
\end{proof}

\begin{proof}[\bf Proof of Theorem \ref{thm1}] This proof is done via classical approximation arguments to show that $\Lambda$ extends to $C^1_0(\r)$ and it is bounded (by the Uniform Boundedness Principle). This way, we can apply Theorem \ref{thm2}. By Remark \ref{rem: z_0} we can assume that $z=i$. We claim that $\Lambda$ extends to a bounded linear functional on $C^1_0(\r)$ with respect to the norm $\|\cdot \|_{1,\infty}$. To this end, let $\eta \in C^\infty_c(\r)$ with $1\geq \eta\geq 0$, $\int \eta = 1$, $\supp(\eta)=[-1,1]$, $\eta (t)=1$ for $|t|<1/4$, and let $\eta_T(x) = T\eta (T x)$ be a standard approximation of identity for $T>0$. Consider the linear functionals
$
\Lambda_T : C^1_0(\r) \to \cp 
$
defined by
$$
\Lambda_T(\p) = \Lambda((\p * \eta_T)T \eta_{1/T}).
$$
We first claim that $\Lambda_T$ is a bounded functional on $C^1_0(\r)$ for each fixed $T>0$. Indeed, if $\p_n\xrightarrow{C^1_0(\r)} \p$ then 
$$
(\p_n * \eta_T)\eta_{1/T}T\to(\p * \eta_T)\eta_{1/T}T \quad \text{in } \ C^1_0(\r).
$$
Let $M>\ep_0^{-1}\sup_{n} \|[(\p_n * \eta_T)\eta_{1/T}T]'\|_{\infty}$,  and notice $\id + (\p_n * \eta_T)\eta_{1/T}T/M \tobl \id + (\p * \eta_T)\eta_{1/T}T/M$. We can now apply \eqref{linear_at_id} and property (P4) to obtain
\begin{align*}
\lim_{n\to\infty} \Lambda_T(\p_n) = \lim_{n\to\infty} M\Lambda_T(\p_n/M) & = \lim_{n\to\infty} M [ E(\id + (\p_n * \eta_T)\eta_{1/T}T/M)(i)-i]  \\ & = M [ E(\id + (\p * \eta_T)\eta_{1/T}T/M)(i)-i] \\ &=  \Lambda_T(\p).
\end{align*}
We concluded that not only $\La_T$ is bounded, but the following representation holds true
$$
\La_T(\p) =  M [ E(\id + (\p * \eta_T)\eta_{1/T}T/M)(i)-i],
$$
for any $M>\ep_0^{-1} \|[(\p * \eta_T)\eta_{1/T}T]'\|_{\infty}$. Now let $\wt{\Lambda}: C^1_0(\r)\to\cp$ be defined by
$$
\wt{\Lambda}(\p) = \lim_{T\to\infty} {\Lambda}_T(\p).
$$
Again, since we have $(\p * \eta_T)\eta_{1/T}T \xrightarrow{C^1_0(\r)} \p$  and $\id+(\p * \eta_T)\eta_{1/T}T/M \tobl \id+\p/M$ as $T\to\infty$, for any $\p\in C^1_0(\r)$, the same procedure now shows that
$
\wt{\Lambda}(\p) =  M [ E(\id + \p/M)(i)-i]
$
for any $\p \in C^1_0(\r)$ and $M>\ep_0^{-1} \|\p'\|_{\infty}$. By the Uniform Boundedness Principle, $\wt \Lambda$ is a  bounded functional on $C^1_0(\r)$ and 
 $$
E(\id + \p)(i)=i+\wt{\Lambda}(\p) 
$$
for any $\p \in C^1_0(\r)$ with $\|\p'\|_{\infty} < \ep_0$.

We can now apply Theorem \ref{thm2} to obtain that $\wt \Lambda=\E_{a,\al}$, for some $a\in \r$ and $\al\geq 0$. The argument above shows that $E(\id +\p)=\E_{a,\al}(\id +\p)$ if $\p\in C_0^1(\R)$ is such that $\|\p'\|_{\infty}<\ep_0$. Let now $\psi_T\in C^\infty_c(\r)$ be such that $\psi_T(x)=1$ for $-T<x<T$, $\|\psi_T\|_\infty\leq 1$, $\supp(\psi_T)\subset [-2T,2T]$ and $|\psi_T'(x)|\leq 1/T$ for $T\leq |x|\leq 2T$. Let $f\in \bl$ be arbitrary such that $\|f'-1\|_\infty<\ep_0/10$, and consider $f_T = \id + (f-\id)\psi_T$. As before, we have $f_T \tobl f$. Since $(f-\id)\psi_T \in C^1_0(\r)$ and $\|f_T'-1\|_\infty < \ep_0$ for large $T$, we obtain $E(f_T)=\E_{a,\al}(f_T)$ for large $T$, which, by property (P4), implies that $E(f)=\E_{a,\al}(f)$. The result now follows from Lemma \ref{lem}.
\end{proof}

\section{Sketch of the proof of Theorem \ref{thm2} in the linear case}\label{sec3}

In order to make the proof of the main theorem easier to understand we will start presenting a sketch for the case of linear extensions.
The strategy of the proof can be stated as follows: (A) Obtain a good representation formula for the extension;
    (B) Use this formula to obtain a functional equation that should be verified by our extension;
    (C) Perform a variational analysis  in this functional equation in order to obtain extra information about the structure of the supporting set of our extension;
    (D) Combine the information contained in the previous steps to obtain our rigidity result. 

It is important to observe that if we start assuming that the extensions are linear, that is, if we assume that there are Radon measures $\{\mu_{y}\}_{y>0}$ such that $\int_\r  (1+|t|)\d|\mu_y|(t)<\infty$ for all $y>0$ and
$$
E(f)(z)=(f*\mu_{y})(x)
$$
for $z=x+iy\in\cp$, the proof is much easier to digest. The admissibility conditions of the class $\A$ reduce to:
\begin{enumerate}
        \item[(L1)]  $\lim_{y\to 0}  f*\mu_y(x) = f(x)$;
        \item[(L2)]  $(a\id+b) * \mu_y(x)=a(x+iy)+b$;
        \item[(L3)]  $(f \circ g)*\mu_y=(f *\mu_{\im [g*\mu_y](x)})(\re [g * \mu_y](x))$ for any $f,g\in\bl$;
\end{enumerate}

\noindent {\bf Step 1.} First we claim that $\d\mu_y(x)=\d\mu(x/y)$. Indeed, letting $f(x)=x+\p(x)$ with compactly supported smooth $\p$ such that $\|\p'\|_\infty<1$, (L3) implies that
$$
ax+aiy + \p(a\cdot)*\mu_y(x) = (x+iy+\p*\mu_y(x))\circ (ax+aiy) = ax + iay + \p*\mu_{ay}(ax)
$$
for $a>0$. This implies that for any such test function $\p$ we have
$$
\int_\r \p(ax-t)\d\mu_y(t/a) = \int_\r \p(ax-t)\d\mu_{ay}(t).
$$
Replacing $x$ by $x/a$ and $y=1
$, this shows that $\d\mu_{1}(t/a)=\d\mu_a(t)$.
Note also that (L2) is now equivalent to 
$$
1=\int_\r \d\mu(t) \text{ and } -i= \int_\r t \d\mu(t).
$$
If we write $\mu=\al-i\be$, then we must have $\int_\r (1,t)\d\al(t)=(1,0)$ and $\int_\r (1,t)\d\be(t)=(0,1)$.

\noindent {\bf Step 2.} Considering $f_j(x)=x+\p_j(x)$ for $j=1,2$ and, as before, $\p_j\in C_c^\infty$ with $\|\p_j'\|_\infty<1$, just note that $E f_j = x+\p_j*\al_y(x) + i(y-\p_j*\be_y(x))$ and $f_1\circ f_2 = x+\p_2(x)+\p_1(x+\p_2(x))$. From $Ef_1 \circ Ef_2 = E (f_1\circ f_2)$ we obtain
\begin{align*}
& x+\p_2*\al_y(x)+\p_1*\al_{y-\p_2*\be_y(x)}(x+\p_2*\al_y(x)) + i(y-\p_2*\be_y(x) -\p_1*\be_{y-\p_2*\be_y(x)}(x+\p_2*\al_y(x))) \\
& = x+\p_2*\al_y(x) + \p_1\circ f_2 *\al_y(x) + i(y-\p_2*\be_y(x)-\p_1\circ f_2 *\be_y(x)).
\end{align*}
After some simple computations, this reduces to
\begin{align*}
\p_1*\ga_{y-\p_2*\be_y(x)}(x+\p_2*\al_y(x)) & = \p_1\circ f_2 *\ga_y(x)
\end{align*}
for any $\ga$ which is a linear combination of $\al$ and $\be$. \\

\noindent {\bf Step 3.} Now replace $\p_2$ by $s\p_2$ and differentiate at $s=0$ to obtain
$$
\int_\r \p_1'(x-ty)\p_2*(\al_y+t\be_y)(x)\d\ga(t)  = \int_\r \p_1'(x-ty)\p_2(x-ty)\d\ga(t).
$$
Put $x=0$ and $y=1$ to conclude that 
$$
\int_\r \p_1'(t)\p_2*(\al-t\be)(0)\d\ga(-t)  = \int_\r \p_1'(t)\p_2(t)\d\ga(-t).
$$
This implies that there is $c\in \r$ such that $\p_2(-t)=\p_2*(\al+t\be)(0) + c$ for $t\in\supp(\ga)$, that is, $\p_2$ is linear in $\supp(\ga)$. Since $\p_2$ is essentially an arbritary test function,  we conclude that  $\#\supp(\ga)\leq 2$.

\noindent {\bf Step 4.} Finally, we find that the measures should be Dirac masses supported at two different points and the constraints in the problem will force that $E=E_{a,\al}$.

\section{Proof of Theorem \ref{thm2}} \label{sec4}
\begin{proof}
\noindent {\bf Step 1 (Integral representation).}  By Remark \ref{rem: z_0} we can assume $z_0=i$.
A routine application of  Riesz's Representation Theorem (the vector-valued form) shows that any bounded linear functional $L: C^1_0(\r)\to \cp $ is given\footnote{Integration by parts show that the measures $\nu$ and $\mu$ are not unique. For instance, replacing $(\d\nu,\d\mu)$ by $(\d\nu+idw,\d\mu+w\d t)$ would represent the same functional.} by
$$
L(\p) = \int \p \d\nu +i \int \p' \d\mu,
$$
where $\nu$ and $\mu$ are complex-valued measures of finite total variation.
In particular, since $\Lambda$ is bounded in $C_0^1(\r)$, we have that
\begin{align}\label{Enewform}
E(\id+\p)(i) = i+  \int_\r \p\d\nu + i\int_\r \p'\d\mu + R(\p) \quad \text{for all } \ \p\in \B(\ep_0).
\end{align}
Properties (P2) and (P3) now imply that for $z=x+iy$ we have
\begin{align}\label{eq:genid}
E(\id+\p)(z) 
& = z + \int_\r \p(x+ty)\d\nu(t) + iy\int_\r \p'(x+ty)\d\mu(t)  + y R(\p_{z})
\end{align}
for $\p\in \B(\ep_0)$, where we define $\p_{z}(t):=y^{-1}\p(x+y t)$.

\noindent {\bf Step 2 (Functional equation derivation).} Now let $\p,\psi\in C^\infty_c(\r)\setminus \{0\}$ and consider
$f_s(x)=x+s\p(x)$ and $g_r(x)=x+r\psi(x)$ for $r,s\in (-\delta,\delta)$, with $\delta=\tfrac{\ep_0}{10\max\{\|\p'\|_\infty,\|\psi'\|_\infty\}}$. Property (P3) implies the identity $E(f_s\circ g_r)(i)=E(f_s)(E(g_r)(i))$, where $f_s\circ g_r(x)=x+r\psi(x)+s\p\circ g_r(x)$. We obtain
\begin{align*}
& E(f_s\circ g_r)(i)  \\ & = i+ \int_\r [r\psi(t)+s\p\circ g_r(t)]\d\nu(t) + i\int_\r [r\psi'(t)+s\p'\circ g_r(t)(1+r\psi'(t))]\d\mu(t) + R(r\psi+s\p\circ g_r)\\ & =  Eg_r(i)+  s\int_\r \p(\re[ Eg_r(i)(1-it)])\d\nu(t) +is \im[ Eg_r(i)]\int_\r \p'(\re[Eg_r(i)(1-it)])\d\mu(t)  \\ & \quad + \im[ Eg_r(i)] R(r\p_z)\\
& = E(f_s)(Eg_r(i)).
\end{align*}
Using that $Eg_r(i) = i + r\int_\r \psi\d\nu + ir\int_\r \psi'\d\mu + R(r\psi)$ we obtain
\begin{align*}
\frac{R(r\psi)-R(r\psi+s\p\circ g_r)-\im[ Eg_r(i)] R(r\p_z)}{rs} = & \int_\r \tfrac1{r}[\p\circ g_r(t)-\p(\re[ Eg_r(i)(1-it)])]\d\nu(t) \\  &   \quad + i\int_\r \tfrac1{r}[\p'\circ g_r(t)(1+r\psi'(t)) \\ & \quad \quad  - \im[ Eg_r(i)]\p'(\re[Eg_r(i)(1-it)])]\d\mu(t).
\end{align*}
We can now set $s=r\to 0$ and use the condition \eqref{linear_at_idweak} to obtain that
\begin{align*}
\int_\r \p'(t)[\psi(t) - (A+Bt)]\d\nu(t) + i\int_\r [\p''(t)\psi(t) + \p'(t)\psi'(t) - \p''(t)(A+Bt) - B\p'(t)]\d\mu(t)  =  0,
\end{align*}
where, if we write $\nu=\nu_1+i\nu_2$ and $\mu=\mu_1+i\mu_2$, we have $A=A(\psi)=\int \psi \d\nu_1-\int \psi' \d\mu_2$ and $B=B(\psi)=\int \psi \d\nu_2+\int \psi' \d\mu_1$. An equivalent form of the last equation is
\begin{align}\label{eq:varidentity}
 \int_\r \p'\eta\d\nu + i\int_\r (\p'\eta)'\d\mu =0,
\end{align}
where $\eta(t)=\eta_\psi(t)=\psi(t)-A(\psi)-B(\psi)t$.

\noindent {\bf Step 3 ($\mu$ and $\nu$ are supported in two-points, modulo integration by parts).} Let now $I$ be a bounded open interval centered at the origin.  Select $\psi \in C^\infty_c(\r)$ such that $\psi(t)=t^2$ for $t\in I$ and $\supp (\psi) \subset 2I$. Let $\Omega=\{\eta_\psi \neq 0\}$, and note that there are $u,v \in \r$ such that $I\setminus \{u,v\} \subset \Omega \cap I$. We claim that $\mu$ is absolutely continuous on $\Omega'=I\setminus \{u,v\}$. Indeed, if we let
$
\d\si = \eta \d\nu +i\eta'\d\mu,
$
then identity \eqref{eq:varidentity} shows that
$$
0 =  \int_\r \p'\d\si + i\int_\r \p''\eta \d\mu = \int_\r \p''(-\si\d t + i\eta\d\mu)
$$
for all $\p\in C^\infty_c(\r)$. Hence there are constants $c_1=c_1(\psi)$ and $c_2=c_2(\psi)$ such that
$i\eta\d\mu = \si\d t + (c_1+c_2t)\d t$. Thus $\mu\in C^1(\Omega')$ and $\mu'(t) = -i (\si(t) + c_1 + c_2t)/\eta(t)$ is of locally bounded variation in $\Omega'$. A calculation shows that
$$
\d\nu-i\d(\mu') = -\frac{c_2}{\eta} \d t
$$
on $\Omega'$. We claim that $c_2=0$. Indeed, since the left hand side above is independent of $\psi$, if $c_2(\psi)\neq 0$ we can select another similar function $\psi_0$, but now with $\psi_0(t)=t^3$ in $I$, to conclude that 
$$
\frac{{t^3-A(\psi_0)-B(\psi_0)t} }{t^2-A(\psi)-B(\psi)t} = \frac{c_2(\psi_0)}{c_2(\psi)}
$$
for all $t\in I$, except, possibly, at $5$ points. But this is impossible. We obtain that 
$\d\nu=i\d(\mu')$ in $\Omega'$. Since $I$ is arbitrary, we conclude that there are $2$ points $u$ and $v$ such that $$\d\nu=i\d(\mu')\quad \text{in} \quad \r\setminus \{u,v\}.$$ This is equivalent to the existence of complex numbers $z_j$ for $j=1,2,3,4$ such that 
\begin{align*}
\d \nu &= z_1\delta_u + z_2\delta_v + i\d(\kappa') \\
\d\mu & = z_3\delta_u + z_4\delta_v + \kappa'\d t
\end{align*}
 for some $\kappa:\r\to\cp$ of class $C^1$, with $\kappa'$ of bounded total variation. Identity \eqref{eq:genid} now shows that 
\begin{align}\label{Enewfunc}
&E(\id+\p)(z)-z  = z_1\p(x+uy) + z_2\p(x+vy) + iy z_3\p'(x+uy) + iyz_4\p'(x+vy) + yR(\p_z)
\end{align}
for all $\p\in \B(\ep_0)$.

\noindent {\bf Step 4 ($\mu$ and $\nu$ have supports and weights related by a linear equation).}  Let $$\E(f)(z) := z_1f(x+uy) + z_2f(x+vy) + iy z_3f'(x+uy) + iyz_4f'(x+vy).$$ Note that $\E|_{C^1_0(\r)}=\Lambda$ and that
$$
E(\id+\p)(z) = z+\E(\p) + yR(\p_z)
$$
for $\p\in \B(\ep_0)$. We claim that
\begin{align}\label{eq:lineareq}
1 & =z_1 +z_2  \\
i & = z_1u + z_2v+iz_3+iz_4.
\end{align}
Notice the above equations are equivalent to $\E(a\id+b)(z)=az+b$ for any $a, b \in \r$, which implies that $E(\id+\p)(z) = \E(\id+\p) + yR(\p_z)$ for $\p\in \B(\ep_0)$. To this end, first notice that by repeating the approximation procedure employed in the proof of Theorem \ref{thm1}, we obtain that for any $f\in \bl$ with $\|f'-1\|_{\infty}<\ep_0/10$ there is $f_n \in \id + \B(\ep_0)$ such that $f_n \tobl f$ as $n\to \infty$.  By condition \eqref{linear_at_idweak}, there is $r_0>0$ such that
$$
|\frac{{E(\id+r(f_n-\id))(z)-z}}{r} -\E(f_n-\id)(z)| = |r^{-1}yR(r(f_n-\id)_z)| \leq ry
$$
if $r \in (0,r_0)$, where $r_0$ is independent of $n$, $z$ and $f$. Letting $n\to\infty$ we conclude that
\begin{align}
|\frac{{E(\id+r(f-\id))(z)-z}}{r} -\E(f-\id)(z)| \leq r y.
\end{align}
Letting $f(x)=(1+\theta)x+b$ for $b\in \r$ and $|\theta| < \ep_0/10$, and using property (P2), we obtain
$$
|\theta z+b -\E(\theta\id+b)(z)| \leq r y.
$$
Finally, letting $r\to 0$ we obtain $\E(\theta\id+b)(z) = \theta z+b$. The desired equations follow.

\noindent {\bf Step 5 (We conclude that $\Lambda={\E_{\max(u,v),|u-v|}.}$).} Assume first that $u=v$. In this case we deduce straightforwardly that 
$$
\E f(z)=f(x+uy)+y(i-u)f'(x+uy) = \E_{u,0}f(z).
$$
Assume now that $u\neq v$. We claim that $z_3=z_4=0$. If this is the case,  equations $z_1+z_2=1$ and $uz_1+vz_2=i$ immediately imply that 
$$
\E=\E_{\max(u,v),|u-v|}.
$$
To this end, going back to equation \eqref{eq:varidentity}, we conclude that 
$$
\p'(u)(z_1\eta(u)+iz_3\eta'(u)) + \p'(v)(z_2\eta(v)+iz_4\eta'(v))+iz_3\p''(u)\eta(u) + iz_4\p''(v)\eta(v) =0
$$
for all $\p,\psi\in C^\infty_c(\r)$ where $\eta=\psi-A-B\id$. Since $u\neq v$, we can prescribe the values $\p'(u), \p'(v), \p''(u), \p''(v)$ freely. Thus, the above equation can only hold if $z_3\eta(u)=z_4\eta(v)=0$ for all $\psi$. Observe that if we let $z_j=a_j+ib_j$ then
\begin{align*}
\begin{bmatrix}
-\eta(u) \\
-\eta(v)
\end{bmatrix} = 
\begin{bmatrix}
a_1+b_1u -1 & a_2+b_2u & -b_3+a_3u & -b_4+a_4u\\
a_1+b_1v -1 & a_2+b_2v & -b_3+a_3v & -b_4+a_4v
\end{bmatrix}
\begin{bmatrix}
\psi(u)\\
\psi(v)\\
\psi'(u)\\
\psi'(v)
\end{bmatrix}.
\end{align*}
We claim that it is not possible that a row in the matrix above is identically zero. For instance, if the second row is zero then $b_3+b_4=a_3v+a_4v$, but the second equation in \eqref{eq:lineareq} implies that $b_3+b_4=a_1u+a_2v$ and $b_1u+b_2v+a_3+a_4=1$. These together imply that  $a_1u+a_2v=v(1- b_1u-b_2v)$. The first equation in \eqref{eq:lineareq} is $a_2=1-a_1$ and $b_2=-b_1$.  Since $a_1+b_1v=1$ we get  $a_2=b_1 v=-b_2 v$. We conclude that  $a_1u=v- b_1vu=v-(1-a_1)u$. Hence, $u=v$; a contradiction. We conclude that $z_3=z_4=0$ as desired. This finishes the proof.
\end{proof}

{\section{Final Remarks}\label{sec5}

\subsection{Alternative version of the condition (P2) - version 1}

We start this final section noting that under condition (P1), condition (P2) is equivalent to:
\begin{enumerate}
\item[(P2*)] $E(f)$ is analytic whenever $f$ is an affine;
\end{enumerate}
Clearly (P2) implies (P2*). In the other direction, let $f(x)=ax+b$ and $F(z)=E f(z)$. Since $\im F> 0$, Poisson representation implies the existence of a finite measure $\mu$, satisfying $$\int_{\r} (1+t^2)^{-1}\d\mu(t)<\infty\,,$$ such that
$$
F(z)=pz+c + \int_\r \frac{1+tz}{t-z}\frac{\d\mu(t)}{1+t^2},
$$
for some $c\in \r$ and $p\geq 0$. We also have that
$$
\mu(t_1)-\mu(t_2) = \lim_{y\to 0} \int_{t_2}^{t_1} \Im F(t+iy)\d t,
$$
whenever $t_1<t_2$ are points of continuity of $\mu$. Since $F(t+iy)\to f(t)$ uniformly in compact sets as $y\to 0$, and $f$ is real-valued, we conclude that $d\mu\equiv 0$, hence $p=a$ and $c=b$, which shows  $F(z)=az+b$.

\subsection{Alternative version of the condition (P2) - Version 2} Other useful observation is that, reasoning in the same way as Epstein and Markovic \cite{Epstein-Markovic}, we could use the rigidity theorem for uniformly quasiconformal groups (obtained by Sullivan \cite{Sullivan} and Tukia \cite{Tukia}), in order to find a quasiconformal mapping $F:\cp\to\cp$ that is the identity on the boundary, and that conjugates the affine group $\aff$ and its image $E({\rm Aff}(\r))$. More precisely, 
    $$
    (F\circ E(\alpha Id+a)\circ F^{-1})(z)=\alpha z+a\,.
    $$
    for any $\alpha>0$ and $a\in\R$.
Since we know that $E({\rm Aff}(\r))$ is uniformly quasiconformal (exploring the Baire theorem for mappings close to the identity), one could simply drop condition (P2) in Conjecture \ref{conj}. We did not use this formulation in first place because the conjugation by $F:\cp\to\cp$ could affect conditions \eqref{linear_at_id} and \eqref{linear_at_idweak}, which are crucial for our techniques to work.

\subsection{Cauchy problem for Norton-Sullivan like extensions}
Observe that each map $f \mapsto \E_{a,\al}f$ solves a partial differential equation (other than the Beltrami equation):  if $f\in C^{2}(\R)$ then $F(x,y)=\E_{a,\al}(f)(x+iy)$ solves the following Cauchy problem:
\begin{equation} 
    \begin{cases}
      &{\rm Tr}(A\,{\rm Hess}(F)) =0\\
      &F(x,0)=f(x) \\
      &\partial_y F(x,0) =if'(x)
    \end{cases}   
    \quad \quad \text{with} \quad \quad A = 
    \begin{bmatrix}
    (\al-a)a & \tfrac12 (2a-\al) \\
    \tfrac12 (2a-\al) & -1
\end{bmatrix}.
\end{equation}
Note that for $(a,\al)=(1,2)$, that is, when $\E_{1,2}=\E_{NS}$ is the Norton-Sullivan extension, this is simply the linear wave equation.

\section*{acknowledgments}
The authors would like to express their gratitude to Dennis Sullivan and Chris Bishop for their remarks and questions during the early stages of this article. F. Gonçalves and L. Oliveira gratefully acknowledge the support of the Office of Naval Research through grant GRANT14201749 (award number N629092412126). F. Gonçalves also acknowledges support from the following funding agencies: The Serrapilheira Institute (Serra-2211-41824), FAPERJ (E-26/200.209/2023) and CNPq (309910/2023-4).

\bibliographystyle{plain}
\bibliography{biblio2}

\end{document}